\documentclass[pdflatex,sn-mathphys-num]{sn-jnl}

\usepackage{graphicx}%
\usepackage{multirow}%
\usepackage{amsmath,amssymb,amsfonts}%
\usepackage{amsthm}%
\usepackage{mathrsfs}%
\usepackage[title]{appendix}%
\usepackage{xcolor}%
\usepackage{textcomp}%
\usepackage{manyfoot}%
\usepackage{booktabs}%
\usepackage{algorithm}%
\usepackage{algorithmicx}%
\usepackage{algpseudocode}%
\usepackage{listings}%

\newtheorem{theorem}{Theorem}
\newtheorem{lemma}{Lemma}%
\newtheorem{remark}{Remark}%

\begin{document}

\title[Two abstract methods of lower and upper solutions with applications]{Two abstract methods of lower and upper solutions with applications}

\author*[1,2]{\fnm{Andrei} \sur{Stan}}\email{andrei.stan@ubbcluj.ro}

\affil*[1]{ \orgname{Department of Mathematics, Babeș-Bolyai
University}, \orgaddress{ \city{Cluj-Napoca}, \postcode{400084}, \country{Romania}}}

\affil[2]{ \orgname{Tiberiu Popoviciu
Institute of Numerical Analysis, Romanian Academy}, \orgaddress{ \city{Cluj-Napoca}, \postcode{400110}, \country{Romania}}}








\abstract{In this paper, we present two abstract methods for constructing a lower and an upper solution for a fixed point equation. The first method applies when the nonlinear operator is a composition of a linear and a nonlinear mapping, while the second method applies when the nonlinear operator satisfies an inequality of Harnack type. An application is provided for each method.
}




\maketitle

\section{Introduction and preliminaries}

The method of upper and lower solutions proves to be extremely useful for solving nonlinear equations, as it not only guarantees the existence of a solution but also provides a localization of the solution within an interval. The literature on this subject is extensive; we mention a few reference works on the topic~\cite{cabada,decoster,precup_lower_upper,1980,jorje}.

The structure of the paper is as follows. In Section~\ref{sectiunea principala}, we present two abstract methods for determining lower and upper solutions for an abstract equation, which are then used to guarantee the existence of a fixed point. The first method concerns an abstract Hammerstein equation (see \cite{precup_lower_upper} for another approach for the same equation), while the second method applies to a fixed point equation in the case where the nonlinear operator satisfies an abstract Harnack-type inequality. We note that the conditions are inspired by~\cite{precup_harnack}, although in that work the method is entirely different, being based on the fixed point index approach. Section~3 is devoted to illustrative applications, each demonstrating one of the two abstract methods.

We conclude this section with two auxiliary results.
The first result is a fundamental theorem from the theory of linear operators, which generalizes the classical Perron–Frobenius theorem for matrices (\cite{Krei Rutman 1948}, see also~\cite[p.~266]{schaefer_topological} or~\cite[Theorem~19.2]{deimling}).
\begin{theorem}[Krein-Rutman]\label{krein rutman}
Let $X$ be a Banach space, $K \subset X$ a total cone, and $F$ a linear compact operator with $F(K) \subset K$ and the spectral radius $r(T)$ strictly positive.  
Then, $r(T)$ is an eigenvalue of $F$ and the corresponding eigenvector lies in the cone $K$.
\end{theorem}
We continue with a variant of Harnack's inequality (see, \cite{trudinger,precup_harnack}). \begin{theorem}\label{trudinger}
    Let $\Omega\subset \mathbb{R}^n$ be a  domain and $D\subset \Omega$ a compact set. Then, for each $p>1$ and $q\in \left[1,\frac{n(p-1)}{n-p}\right] $ there exists a constant $\mathcal{C}>0$ such that \begin{equation*}
        \inf_Du\geq \mathcal{C}\left ( \int_D u^q\right)^{\frac{1}{q}},
    \end{equation*}
    for all $u$ such that $u\geq 0$ and $-\Delta_pu\geq 0.$
\end{theorem}

\section{Main abstract results}\label{sectiunea principala}
Let $X$ be a Banach space with norm $|\cdot|$, ordered by a cone $K$.  
The induced order relation given by the cone $K$ is denote by $\leq$, that is, for $u,v \in X$ we write $u \leq v$ if and only if $v-u \in K$. Throughout this section, we always assume that the norm $|\cdot|$ is semi-monotone, i.e., there exists $\gamma>0$ such that \begin{equation}\label{norma semi monotona}
    0\leq u\leq v \quad \text{ implies }\quad |u|\leq \gamma|v|.
\end{equation}

Each of the following two subsections presents a different method for constructing lower and upper solutions for a fixed point equation.  
The first approach concerns  an abstract Hammerstein equation.

\subsection{Abstract Hammerstein equations}\label{prima sectiune}

We consider the equation  \begin{equation}\label{ecuatie abstracta}
    u=LF(u),
\end{equation}
where   $L \colon X \to X$ is a linear operator and $ F \colon K \to K$ is a (nonlinear) continuous mapping. In the subsequent we denote $$N:=LF.$$
Our first condition is related to the linear operator $L$.
\begin{description}
\item[(h1)] The cone $K$ is invariant under the linear operator $L$, i.e., $L(K)\subset K$. Moreover, $L$ admits a positive eigenvalue $\lambda_1 > 0$ with the corresponding eigenfunction $\varphi_1$ from the cone $K$ ($\varphi_1 \in K$).
\end{description}

From (h1) it follows that $L$ is increasing (order-preserving), that is, for all $u, v \in K$, one has
\[
  0 \leq u \leq v \quad \text{implies} \quad Lu \leq Lv.
\]
Indeed, if $0 \leq u \leq v$, then $v - u \in K$. Since $L(v - u) \in K$, i.e., $L(v - u) \geq 0$, it follows that $Lv - Lu \geq 0$, which proves the claim.

Under certain conditions on $ L$ and $F$, we show that  there exists a \textit{lower solution} $\underline{u}$ , i.e.,\begin{equation*}
    \underline{u}\leq N  \underline{u}.
\end{equation*}
and an \textit{upper solution}
\begin{equation*}
     \overline{u}\geq N \: \overline{u},
\end{equation*}
for the equation \eqref{ecuatie abstracta}. 

Before stating our assumptions, let \( \Phi \colon K \to \mathbb{R}_+ \) be a positively homogeneous mapping, that is,  
\[
    \Phi(\alpha u) = \alpha \Phi(u) \quad \text{for all } \alpha > 0, \; u \in K,
\] and
with the additional property that \( \Phi(u) = 0 \) if and only if \( u = 0 \). 
 \begin{remark}
The mapping \( \Phi \) can, for instance, to be a seminorm.
 \end{remark}
    
    The following conditions are assumed to be satisfied:
\begin{description}
\item[(h2)] The operator $F$ is increasing (order-preserving), that is, for $u,v \in K$ one has 
\begin{equation*}
  0\leq  u \leq v \quad \text{implies} \quad F(u) \leq F(v).
\end{equation*}
    \item[(h3)]\label{cond r} There exists a constant $r > 0$ such that
\[
    F(\lambda_1 u) \geq u \quad \text{for all } u \in K \text{ with } \Phi(u) = r .
\]
    \item[(h4)] There exists $\alpha>0$ and $\mu \in K \setminus \{0\}$  such that \begin{equation*}
        L\mu \leq\alpha \mu
    \end{equation*}
    and  \[
        F(\alpha\mu) \leq \mu .
    \]
\end{description}
Denote  \begin{equation}
\label{definire sub si supra solutii}
  \underline{u}=\frac{r}{\Phi(\varphi_1)}L(\varphi_1) \quad
  \text{ and }\quad \overline{u}=L\mu.
\end{equation}
We now show that $\underline{u}$ and $\overline{u}$ are a lower and an upper solution, respectively, for the equation~\eqref{ecuatie abstracta}.
\begin{theorem}\label{teorema existenta lower si upper solutie}
  Assume that  (h1)–(h4) hold true. Then, \(\underline{u}\) is a lower solution and \(\overline{u}\) is an upper solution for the equation~\eqref{ecuatie abstracta}.
\end{theorem}
\begin{proof}
To show that $\underline{u}$ is a lower solution, first observe that since 
\[
   \Phi\left(  \frac{r}{\Phi(\varphi_1)}\varphi_1\right) =   \frac{r}{\Phi(\varphi_1)}\Phi\left(\varphi_1\right)=r,
\] 
and $\varphi_1\in K\setminus\{0\}$, condition (h3) implies
\[
      \frac{r}{\Phi(\varphi_1)}\varphi_1 
    \leq T\!\left(\frac{\lambda_1 r}{\Phi(\varphi_1)}\varphi_1 \right) 
    = T\!\left(\frac{r}{\Phi(\varphi_1)}L\!\left(\varphi_1\right)\right) 
    = F(\underline{u}).
\]
By the order-preserving property of $L$, it follows that
\[
   \underline{u} = L\!\left(\frac{r}{\Phi(\varphi_1)}\varphi_1\right) \leq LF(\underline{u})=N(\underline{u}),
\]
as desired.

Concerning the function $\overline{u},$  using the monotonicity property of $F$ from (h2), together with (h4), one obtains
\[
    \mu \geq F(\alpha \mu) \geq F(L\mu) = F(\overline{u}).
\]
Applying again the order-preserving property of $L$ to this relation yields
\[
    \overline{u} = L\mu \geq N(\overline{u}),
\]
 hence $\overline{u}$ is an upper solution.
\end{proof}
Further, assume that  
\begin{description}
    \item[(h5)] The lower and upper solutions $\underline{u}$ and $\overline{u}$ satisfy (are comparable)
    \[
        \underline{u} \leq \overline{u}.
    \]
\end{description}
Then the following invariance result holds.
\begin{theorem}\label{teorema invarianta}
Under conditions (h1)-(h5),  the interval $[\underline{u},\overline{u}]$ is invariant under the operator $N$, that is,
\begin{equation*}
   N\big( [\underline{u},\overline{u}] \big)\;\subset\; [\underline{u},\overline{u}].
\end{equation*}
\end{theorem}
\begin{proof}
Let $u \in [\underline{u}, \overline{u}]$. By the monotonicity of $L$ and $F$, we have
\[
    \underline{u} \leq u \leq \overline{u} 
    \quad \text{ implies } \quad 
   N(\underline{u})=  LF(\underline{u}) \leq LF(u) \leq LF(\overline{u})=N(\overline{u}).
\]
Since $\underline{u}\leq N(\underline{u})$ and $\overline{u}\geq N(\overline{u})$, it follows that
\[
    \underline{u} \leq Nu \leq \overline{u},
\]
which completes our proof.
\end{proof}
\begin{remark}\label{remarca verificare h4}
    One easily sees that condition (h5) holds if \begin{equation*}
        \frac{r}{\Phi(\varphi_1)}\varphi_1\leq \mu,
    \end{equation*}
    which proves to be useful in applications, as we are about to see in the next section.
\end{remark}
Under an additional compactness condition on $L$, we obtain the following fixed-point result.
\begin{theorem}\label{teorema existenta solutie}
Assume that conditions (h1)-(h5) are satisfied. If, in addition, the operator $L$ is completely continuous, then 
there exists $u^\ast \in [\underline{u},\overline{u}]$ such that 
\[
    u^\ast = N (u^\ast).
\]
\end{theorem}
\begin{proof}
Clearly, the set $[\underline{u}, \overline{u}]$ is convex. Since the cone $K$ is closed, it follows that $[\underline{u}, \overline{u}]$ is also closed. Moreover, as the norm $|\cdot|$ is semi-monotone, for all $u \in [\underline{u}, \overline{u}]$ one has
\[
    |u| \leq \gamma |\overline{u}|,
\]
and hence the set $[\underline{u}, \overline{u}]$ is bounded.

The complete continuity of \( L \), together with the continuity and monotonicity of \( F \), implies that the operator \( N \) is continuous and maps the set \( [\underline{u}, \overline{u}] \) into a relatively compact set. 
Indeed, since \( F \) is increasing and the norm is semi-monotone, one obtains that \( F([\underline{u}, \overline{u}]) \) is bounded, more exactly, $|F(u)| \leq \gamma\, |F(\overline{u})|$ for all $u\in [\underline{u}, \overline{u}]$. Thus, by the compactness of \( L \), it follows that \( N([\underline{u}, \overline{u}]) \) is relatively compact, which establishes the desired property.

Finally, by Theorem~\ref{teorema invarianta}, the interval \( [\underline{u}, \overline{u}] \) is invariant under the operator \( N \). Hence, Schauder’s fixed point theorem applies and ensures the existence of a fixed point \( u^\ast \in [\underline{u}, \overline{u}] \) for \( N \), which finishes our proof.

\end{proof}

\begin{remark}
From the proof of Theorem~\ref{teorema existenta lower si upper solutie},  
we observe that instead of $\varphi_1$ and $\lambda_1$, one may take any nonzero element  
$\varphi \in K\setminus\{0\}$ and any $\lambda > 0$ such that  
\[
    L(\varphi) \geq \lambda \varphi.
\]
In this case, $\lambda_1$ in (h3) should be replaced by $\lambda$, and $\underline{u}$ to be defined by
\begin{equation}\label{underline u mai general}
    \underline{u} = \frac{r}{\Phi(\varphi)}\, L(\varphi).
\end{equation}
Indeed, since
\[
    \frac{r}{\Phi(\varphi)}\varphi 
    \leq T\!\left(\frac{\lambda r}{\Phi(\varphi)}\varphi\right) 
    \leq T\!\left(\frac{r}{\Phi(\varphi)}L(\varphi)\right) 
    = F(\underline{u}),
\]
the order-preserving property of $L$ implies that $\underline{u}$ given in \eqref{underline u mai general} is a lower solution for the equation \eqref{ecuatie abstracta}.

\end{remark}
If, instead of conditions (h3) and (h4), we consider
\begin{description}
    \item[(h3)'] There exists $R > 0$ such that
    \[
        F(\lambda_1 u) \leq u \quad \text{for all } u \in K \text{ with } \Phi(u) = R,
    \]

    \item[(h4)'] There exist $\alpha > 0$ and $\mu \in K \setminus \{0\}$ such that
    \[
        L\mu \geq \alpha \mu
    \]
    and
    \[
        F(\alpha \mu) \geq \mu,
    \]
\end{description}
an analogue of Theorem \ref{teorema existenta lower si upper solutie} can be established.
\begin{theorem}
Assume that conditions (h1), (h2), (h3)' and (h4)' are satisfied. Then
\[
    \underline{u} = L\mu 
    \quad \text{and} \quad 
    \overline{u} = \frac{ R}{\Phi(\varphi_1)}\,L(\varphi_1),
\]
are a lower and an upper solution, respectively, for the equation~\eqref{ecuatie abstracta}.
\end{theorem}

\begin{proof} 
Under similar reasoning as in the proof of Theorem \ref{teorema existenta lower si upper solutie}, we obtain
\begin{equation*}
    \underline{u} = L\mu \leq LF(\alpha \mu) \leq LF\left(L\mu \right) = N(\underline{u}),
\end{equation*}
and
\begin{equation*}
    \overline{u} 
    = L\left(\frac{R}{|\varphi_1|}  \varphi_1\right) 
    \geq LF\!\left(\frac{\lambda_1 R}{\Phi(\varphi_1)}\varphi_1\right) 
    =LF\left( \frac{R}{\Phi(\varphi_1)}L(\varphi_1)\right) =N(\overline{u}).
\end{equation*}
\end{proof}
\begin{remark}
We emphasize that, under the same assumptions as in Theorem \ref{teorema existenta solutie}, the method of monotone iterations can be applied to obtain two extremal (not necessarily distinct) fixed points for the operator $N$. That is, there exist 
\(\underline{u}^\ast, \overline{u}^\ast \in [\underline{u}, \overline{u}]\) such that 
\(N(\underline{u}^\ast) = \underline{u}^\ast\) and \(N(\overline{u}^\ast) = \overline{u}^\ast\), 
and every other fixed point of \(N\) from $[\underline{u},\overline{u}]$ lies in the interval 
\([\underline{u}^\ast, \overline{u}^\ast]\). 
We refer the reader to~\cite{decoster} or~\cite{precup_carte} for further details on the monotone iterative method.

\end{remark}
In the next section, we present another method for constructing a lower and an upper solution for a fixed point equation, applicable when the nonlinear operator satisfies an abstract Harnack inequality.
\subsection{Fixed point equations via abstract Harnack inequality}

We consider the fixed point problem
\begin{equation}\label{problema punct fix}
    u = N(u),
\end{equation}
where $N \colon X \to X$ is a (nonlinear)  increasing operator, i.e., for $u,v \in K$ one has
\[
  0 \leq u \leq v \quad \text{implies} \quad N(u) \leq N(v).
\]
On $X$, we consider a seminorm $\|\cdot\|$, which is assumed to be increasing, i.e., for $u,v \in K$ one has
\[
    0 \leq u \leq v \quad \text{implies} \quad \|u\| \leq \|v\|.
\]
Moreover, we assume that there exists $\psi \in K\setminus\{0\}$ such that
\begin{equation}\label{inegalitate superioara psi}
     u \leq |u|\,\psi \quad \text{for all } \,\,u \in K.
\end{equation}
The following condition plays a key role in the subsequent analysis and can be regarded as a weak type \textit{Harnack inequality} \cite{precup_harnack}.
\begin{description}
    \item[(a1)] There exists $\chi \in K \setminus \{0\}$ such that
\begin{equation}\label{inegalitate Harnack}
    N(u) \geq \|N(u)\|\chi \quad \text{for all } u \in K.
\end{equation}
\end{description}
The next  two additional conditions are required.
\begin{description}
    \item[(a2)] There exists $r > 0$ such that
    \[
        \|N(r\chi)\| \geq r.
    \]

    \item[(a3)] There exists $R > 0$ such that
    \[
        |N(R\psi)| \leq R.
    \]
\end{description}

Now, we are ready to present the main result of this subsection.
\begin{theorem}\label{teorema principala a doua sectiune}
Assume that conditions (a1)-(a3) are satisfied. Then
\[
    \underline{u} =r\chi 
    \quad \text{and} \quad 
    \overline{u} = R \psi,
\]
represent a lower and an upper solution, respectively, for the equation \eqref{problema punct fix}. Moreover, if\begin{equation}\label{u lowerbar < u upperbar}
    \underline{u}\leq \overline{u},
\end{equation} and the operator $N$ maps the interval $[\underline{u},\overline{u}]$ into a relatively compact set,   then $N$ admits a fixed point in the interval $[\underline{u}, \overline{u}]$, i.e., there exists $u^\ast \in K$ such that
\begin{equation}\label{fixed point}
     \underline{u} \leq u^\ast \leq \overline{u}
    \quad \text{and} \quad 
    N(u^\ast) = u^\ast.
\end{equation}
\end{theorem}
\begin{proof}
Note that, since the seminorm \( \|\cdot\| \) is increasing, and using the Harnack inequality~\eqref{inegalitate Harnack} together with condition~(a2), we obtain
\begin{equation}\label{u underline sub solutie}
    N(\underline{u}) = N(r\chi) \ge \|N(r\chi)\|\chi \ge r\chi = \underline{u},
\end{equation}
so \( \underline{u} \) is a lower solution of problem~\eqref{problema punct fix}. Moreover, for any \( u \ge \underline{u} \), by the monotonicity of \( N \) and relation~\eqref{u underline sub solutie}, we have
\begin{equation}\label{sub solutie}
    N(u) \ge N(\underline{u}) \ge \underline{u}.
\end{equation}
Taking $u=N(R\psi)$ in \eqref{inegalitate superioara psi}, we have  \begin{equation*}
    N(\overline{u})=N(R\psi)\leq |N(R\psi)|\psi=|N(\overline{u})|\psi.
\end{equation*}
Thus, by (a3), one obtains \begin{equation*}
 N(\overline{u})\leq |N(\overline{u})|\psi=|N(R\psi)|\psi\leq R\psi=\overline{u},
\end{equation*}
whence $\overline{u}$ is an upper solution for the problem \eqref{problema punct fix}.
For any $u\leq \overline{u}$, 
by the monotonicity of $N$, we further obtain that \begin{equation}\label{supra solutie}
    N(u)\leq N(\overline{u})\leq \overline{u}.
\end{equation}
Assume now that relation~\eqref{u lowerbar < u upperbar} holds. Then,  from~\eqref{sub solutie} and~\eqref{supra solutie} it follows immediately that  
\[
    N\big([\underline{u}, \overline{u}]\big) \subset [\underline{u}, \overline{u}].
\]
Moreover, the set \( [\underline{u}, \overline{u}] \) is convex, closed, and bounded (the boundedness follows from the semi-monotonicity of the norm). Now, if $N$ maps $[\underline{u}, \overline{u}]$ into a relatively compact set,  then the Schauder’s fixed point theorem applies and guarantees that there exists a fixed point for \( N \) in the interval \( [\underline{u}, \overline{u}] \), that is, relation~\eqref{fixed point} holds.
\end{proof}
\begin{remark}
By the monotonicity of the seminorm \( \|\cdot\| \) and the Harnack inequality~\eqref{inegalitate Harnack}, we have \( \|\chi\| \le 1 \). Consequently, compared to condition~(2.18) in~\cite{precup_harnack}, assumption~(a2) represents a weaker requirement. Indeed, if condition~(2.18) in~\cite{precup_harnack} holds, that is,
\[
    \|N(r \chi)\| \ge \frac{r}{\|\chi\|},
\]
then
\[
    \|N(r \chi)\| \ge \frac{r}{\|\chi\|} \ge r,
\]
and hence assumption~(a2) is also satisfied.

\end{remark}

\section{Applications}
In this section, we present one application for each of the two abstract results from Section~\ref{sectiunea principala}.
\subsection{Positive solution of the classical Hammerstein equation}
In this subsection, we consider the fixed point problem
\begin{equation}\label{ec integrala Hammerstein}
     u(t) = \int_0^1 k(t,s) f(u(s)) \, ds,
\end{equation}
for which we show how Theorem~\ref{teorema existenta solutie} can be applied. Here, $k$ is a symmetric ($k(t,s)=k(s,t)$ for all $t,s\in [0,1]$), nonnegative continuous function on $[0,1]^2$, and $f \in C(\mathbb{R}, \mathbb{R}_+)$ is nondecreasing on $\mathbb{R}_+$.  

Let $X = C[0,1]$ be endowed with the supremum norm $|\cdot|_\infty$, and let $K$ denote the cone of continuous nonnegative functions.  
Also, let $L\colon X\to X$ be the linear Hammerstein operator
\[
    (Lu)(t) = \int_0^1 k(t,s) u(s) \, ds, \quad u \in C[0,1],
\]
and $F \colon K \to K$ the Nemytskii operator, which assigns to each $u \in K$ the function $Fu \colon [0,1] \to \mathbb{R}_+$ given by
\[
   F(u)(s) = f(u(s)), \quad s \in [0,1].
\]
Note that $F$ is well defined since $f$ is continuous and nonnegative on $\mathbb{R}$.

The complete continuity of $L$ follows from standard arguments based on the Arzelà–Ascoli theorem (see, e.g., \cite{integration theory for hammerstein operators, liu2009, cheng2019, lou1997}). To guarantee that $L$ has an eigenvalue, we need the following additional condition on the kernel $k$ (see \cite{cheng2019}). 
\begin{description}
    \item[(H1)] The function $k$ satisfies a Green like inequality, that is, there exists a continuous function $\theta\colon[0,1]\to \mathbb{R}_+$ such that \begin{equation*}
        k(t,s)\geq \theta(t)k(q,s)\quad \text{for all}\quad t,s,q\in [0,1],
    \end{equation*}
    and  \begin{equation*}
      |L\theta|_\infty>0.
    \end{equation*}
\end{description}
The following result is of great importance for the subsequent analysis.
\begin{lemma}\label{lema existena valori proprii hammerstein}
If condition (H1) is satisfied, then the linear Hammerstein operator $L$ has a positive eigenvalue $\lambda_1>0$, and moreover, the corresponding eigenfunction $\varphi_1$ is positive, i.e., $\varphi_1\in K$.
\end{lemma}
\begin{proof}
The proof relies on the Krein–Rutman theorem (Theorem~\ref{krein rutman}).  
If one can show that the spectral radius $r(L)$ of $L$ is strictly positive, then, since $L$ is completely continuous and the cone $K$ is reproducing (see \cite[Chapter~19]{deimling}), and hence total, Theorem~\ref{krein rutman} applies, and the conclusion follows.

  Now we show that $r(L) > 0$.  
To this end, we use Gelfand’s formula (see, e.g., \cite[p.~79]{deimling})
\[
    r(L) = \lim_{k \to \infty} \left| L^k \right|_{\text{op}}^{1/k},
\]
where $|\cdot|_{\text{op}}$ denotes the operator norm, i.e., \begin{equation*}
    |L|_{\text{op}}=\sup_{u\neq 0}\frac{|Lu|_\infty}{|u|_\infty}.
\end{equation*}  
Simple computations yields that 
\[
    (L\theta)(t) = \int_0^1 k(t,s)\theta(s)\,ds 
    \geq \theta(t) \int_0^1 k(q,s)\theta(s)\,ds 
    = \theta(t)\,(L\theta)(q),
\]
for all $q\in [0,1]$.
Hence,
\[
    (L\theta)(t) \geq \theta(t)\,|L\theta|_\infty \quad \text{for all } t \in [0,1].
\]
Thus, for any $k \in \mathbb{N}$, one has
\[
    (L^k\theta)(t) \geq \theta(t)\,|L\theta|_\infty^k.
\]
Taking the supremum norm, we obtain
\[
    |L^k\theta|_\infty \geq |\theta|_\infty\,|L\theta|_\infty^k.
\]
Consequently, since
\[
    |L^k|_{\text{op}} \geq \frac{|L^k\theta|_\infty}{|\theta|_\infty} 
    \geq |L\theta|_\infty^k,
\]
it follows that 
\[
    r(L) \geq |L\theta|_\infty > 0,
\]
where the latter inequality follows by condition (H1).
\end{proof}
\begin{remark}
Other conditions than {(H1)} can ensure that the conclusion of Lemma $\text{\ref{lema existena valori proprii hammerstein}}$ remains valid,  for instance,  $\min_{t\in [0,1]}k(t,t)>0$ (see \cite[Lemma 1]{lou1997}).
\end{remark}
By Lemma~\ref{lema existena valori proprii hammerstein}, the operator~$L$ has a positive eigenvalue $\lambda_1 = r(L) > 0$ with the corresponding eigenfunction $\varphi_1 \in K$, hence condition~(h1) is fulfilled. Moreover, since the function~$f$ is continuous and nondecreasing, the  Nemytskii operator~$F$ is continuous and increasing, therefore, condition~(h2) is also satisfied.

To ensure that {(h3)} and {(h4)} are valid as well, the following asymptotic conditions on $f$ are required:
\begin{description}
    \item[(H2)] One has, \begin{equation*}
        \lim_{t\searrow 0} \frac{f(t)}{t}>\frac{1}{\lambda_1} \quad \text{ and }\quad \lim_{t\to \infty} \frac{f(t)}{t}<\frac{1}{|L1|_\infty},
    \end{equation*} 
    where \begin{equation*}
        (L1)(t)=\int_0^1 k(t,s)ds.
    \end{equation*}
\end{description}
Note that, by (H2), there exists $0<r<R$  such that \begin{equation}\label{f(t)<=lambda_1t}
    f(t)\geq \frac{t}{\lambda_1} \quad \text{ for all}\,\, t\in [0,\lambda_1r],
\end{equation}
and \begin{equation}\label{f(t)<t supra L1 infinit}
    f(t)\leq \frac{t}{|L1|_\infty} \quad \text{ for all}\,\, t\geq |L1|_\infty R.
\end{equation}
Now, the following result holds. 
\begin{theorem}
    Under conditions (H1) and (H2), there exists a positive solution $u^\ast\in K$ for the problem \eqref{ec integrala Hammerstein}, and moreover, \begin{equation*}
         \frac{r}{|\varphi_1|_\infty}(L\varphi_1)(t)\leq u^\ast (t) \leq (LR)(t) \quad \text{ for all}\,\, t\in [0,1].
    \end{equation*} 
\end{theorem}
\begin{proof}
   We verify that all the assumptions of Theorem~\ref{teorema existenta solutie} are satisfied.  Conditions (h1) and (h2) are valid, as explained above.

    \textit{Check of condition (h3)}. Letting $\Phi(u)=|u|_\infty$, we see that for any $u\in K$ with $\Phi(u)=r$, one has $\lambda_1u(t)\leq \lambda_1r$ ($t\in [0,1]$). Thus, using \eqref{f(t)<=lambda_1t}, we deduce that   \begin{equation*}
        f(\lambda_1u(t))\leq \frac{\lambda_1u(t)}{\lambda_1}=u(t) \quad \text{ for all }\,\,t\in [0,1],
    \end{equation*}
    so condition (h2) is verified.

     \textit{Check of condition (h4)}. Let $\mu\equiv R$. One has,  \begin{equation*}
         (L\mu)(t)=(LR)(t)=R(L_1)(t)\leq \alpha R\quad (t\in [0,1]),     \end{equation*}
         where $\alpha=|L1|_\infty$, and 
         \begin{equation*}
             F(\alpha\mu)=f(R|L1|_\infty)\leq \frac{R|L1|_\infty}{|L1|_\infty}=\mu,
         \end{equation*}
         whence condition (h4) holds.

           \textit{Check of condition (h5)}. Based on Remark \ref{remarca verificare h4}, condition (h5) is satisfied if 
           \begin{equation*}
               \frac{r}{|\varphi|_\infty}\varphi_1(t)\leq \mu=R,
           \end{equation*}
          which is clearly true since $r \le R$.

          Therefore, Theorem \ref{teorema existenta solutie} applies and gives the conclusion.
\end{proof}

\subsection{Positive solutions of  $p$-Laplace equations}
In  this subsection, inspired by \cite{precup_harnack}, we apply Theorem \ref{teorema principala a doua sectiune} for the $p$-Laplace problem \begin{equation}\label{p laplace problem}
    \begin{cases}
        -\Delta_pu=f(u) \quad \text{in }\,\, \Omega\\
        u=0 \quad \text{on }\,\, \partial\Omega\\
        u>0,
    \end{cases}
\end{equation}
where $p>1$, $\Omega$ is a smooth domain in $\mathbb{R}^n$ and $f\colon \mathbb{R}_+\to \mathbb{R}_+$ is a nondecreasing continuous function. Following~\cite[Lemma~1.1]{az} (see also \cite{djm}), for each \( h \in L^{\infty}(\Omega) \), there exists a unique (weak) solution \( S(h) \in C_0^1(\overline{\Omega}) \) to the problem  
\[
   \begin{cases}
        -\Delta_p u = h \quad \text{ in } \Omega,\\[4pt]
        u \in W^{1,p}_0(\Omega).
   \end{cases}
\]
Moreover, the operator \( S \colon L^{\infty}(\Omega) \to C_0^1(\overline{\Omega}) \) is completely continuous, order preserving (increasing), and invariant with respect to the cone of positive functions, i.e., \( S(h) \ge 0 \) whenever \( h \ge 0 \). Thus, letting \( X = L^{\infty}(\Omega) \) be endowed with the supremum norm $|\cdot|_\infty$ (hence increasing), and denoting by \( K \) the cone of positive functions, the problem~\eqref{p laplace problem} allows for the fixed point formulation  
\begin{equation}\label{ecuatie p laplace}
     u = N (u), \quad u \in K,
\end{equation}
where \( N = S F \), and \( F \) is the Nemytskii operator that associates to each function \( u\in L^{\infty}(\Omega) \) the function  
\[
    F(u)(x) = f(u(x)), \quad x \in \Omega.
\]
Note that, since $f$ is continuous and nonnegative, the operator $F$ is continuous and $F(K)\subset K$. 

We immediately observe that relation~\eqref{inegalitate superioara psi} holds with \( \psi \equiv 1 \). Moreover, given that the function \( f \) is nondecreasing, the operator \( T \) is increasing, and because \( S \) is order-preserving, it follows that \( N \) is also increasing.

By Theorem \ref{trudinger}, for some fixed compact set $D\subset \Omega$, there exists $M>0$ such that for every $p$-superharmonic function $u\in K$, one has\begin{equation*}
    \inf_{D}u\geq M\int_D u(x)dx.
\end{equation*}
Let us consider the seminorm \( \|\cdot\| \) be given by  
\[
    \|u\| = M \int_D u(x)\,dx, \quad u \in L^{\infty}(\Omega).
\]
Then, since \( S(h) \) is nonnegative and \( p \)-superharmonic for each \( h \in K \), condition  (a1) (inequality ~\eqref{inegalitate Harnack}) holds with  
\[
    \chi(x) =
    \begin{cases}
        1, & x \in D,\\[4pt]
        0, & \text{otherwise}.
    \end{cases}
\]
Note for any $q>0$, one has (see \cite{precup_harnack})\begin{equation*}
    Sq=q^{\frac{1}{p-1}}S1.
\end{equation*} 
Thus,  it follows that \begin{equation*}
 \|N(q\chi)\| =  \|SF(q\chi)\|=f(q)^{\frac{1}{p-1}}\|S1\|
\end{equation*}
and \begin{equation*}
    |N(q\psi)|_\infty=|N(q)|_\infty\leq f(q)^\frac{1}{p-1}|S1|_\infty.
\end{equation*}
If \begin{description}
     \item[(A1)] There exists $0<r\leq R$ such that \begin{equation*}
    f(r)\geq \frac{r^{p-1}}{\|S1\|^{p-1}},
\end{equation*}
and \begin{equation*}
    f(R)\leq \frac{r^{p-1}}{|S1|_\infty^{p-1}},
\end{equation*}
 \end{description}
then conditions (a2) and (a3) are verified. In addition, since $r\leq R$ then relation \eqref{u lowerbar < u upperbar}  is also valid, where    \begin{equation}\label{underline u si overline u aplicatie p laplace}
       \underline{u}:=r\chi \quad \text{ and }\quad \overline{u}=R.
   \end{equation} 
We easily see that, since \( F \) is continuous, the operator \( N \) is also continuous. Moreover, as \( F \) is increasing, the set \( F([\underline{u}, \overline{u}]) \) is bounded. Hence, the complete continuity of \( S \) implies that the set \( N([\underline{u}, \overline{u}]) \) is relatively compact.

   Therefore, all the requirements of Theorem \ref{teorema principala a doua sectiune} are satisfied, so the following result holds. 
\begin{theorem}
   Assume that condition (A1) is satisfied. Then,
$\underline{u}$ and $\overline{u}$ given in \eqref{underline u si overline u aplicatie p laplace}
   represent a lower and upper solution, respectively, for the equation \eqref{ecuatie p laplace}. Moreover, 
   there exists $u^\ast\in K$ such that it is a solution for the problem \eqref{p laplace problem} and \begin{equation*}
        r\chi 
\leq u^\ast\leq
 R.
    \end{equation*}
\end{theorem}

\bmhead{Acknowledgements}
The author would like to thank Prof. Radu Precup for insightful discussions on the subject and  verification of the paper.

\section*{Declarations}

\subsection*{Funding}
The authors declare that no funds, grants, or other support were received during the preparation of this manuscript.
\subsection*{Conflict of interest}
The author have no relevant financial or non-financial interests to disclose.

\subsection*{Data availability }
Not applicable

\subsection*{Ethics approval and consent to participate}
Not applicable
\subsection*{Materials availability}
Not applicable
\subsection*{Code availability}
Not applicable





\end{document}